\documentclass[12pt, oneside]{article}

\usepackage[T1]{fontenc}
\usepackage[utf8]{inputenc}
\usepackage[a4paper, total={6.2in, 9.1in}, left={1.3in}, right={1.3in}, top={1.1in},bottom={1.3in}, footskip={.5in}]{geometry}
\usepackage{amstext,amsmath, amssymb, amsthm, mathrsfs, mathtools}
\usepackage{thmtools}
\usepackage{titlesec}
\usepackage{authblk}
\usepackage{enumerate}
\usepackage[dvipsnames]{xcolor}
\usepackage{comment}
\usepackage{etoolbox}
\usepackage{stmaryrd}

\usepackage{setspace}
\allowdisplaybreaks

\makeatletter
\newcommand\@authorinfo{}
\renewcommand{\author}[2][]{
  \g@addto@macro\@author{#2\and}
  \g@addto@macro\@authorinfo{
    \par\bigskip
    \noindent{\scshape #2}\par
  }
}
\newcommand{\address}[1]{
  \g@addto@macro\@authorinfo{#1\par}
}
\newcommand{\email}[1]{
  \g@addto@macro\@authorinfo{\texttt{#1}\par}
}
\AtEndDocument{
  \par\bigskip
  \@authorinfo
}
\makeatother

\makeatletter
\renewcommand{\maketitle}{
    \vspace*{1em}
    \begin{center}
        {\normalfont\bfseries\MakeUppercase{\@title}\par}
        \vskip 1em
        {\normalfont\small \@date \par}
    \end{center}
}
\makeatother

\renewenvironment{abstract}{
    \noindent
    \begin{list}{}{
        \setlength{\leftmargin}{2.5em}
        \setlength{\rightmargin}{2.5em}
        }
        \item\relax
        \textsc{Abstract. }\ignorespaces
        }{
  \end{list}
}

\titleformat{\section}
  {\normalfont\scshape\centering}
  {\thesection.}{0.5em}{}

\titleformat{\subsection}[runin]
  {\normalfont\bfseries}
  {\thesubsection.}{0.5em}{}[.]

\usepackage[backend=bibtex]{biblatex}
\renewbibmacro{in:}{}
\DeclareFieldFormat{pages}{#1}

\newtheorem{theorem}{Theorem}[section]
\newtheorem{lemma}[theorem]{Lemma}

\newtheorem{proposition}[theorem]{Proposition}
\newtheorem{corollary}[theorem]{Corollary}

\theoremstyle{definition} 

\numberwithin{equation}{section}

\DeclareMathOperator{\nR}{\mathbb{R}}

\DeclareMathOperator{\nN}{\mathbb{N}}

\DeclareMathOperator{\supp}{supp}
\DeclareMathOperator{\Id}{Id}

\DeclareMathOperator{\Rm}{Rm}
\DeclareMathOperator{\Ric}{Ric}
\DeclareMathOperator{\R}{R}

\DeclareMathOperator{\Div}{div}

\DeclareMathOperator{\B}{B}
\DeclareMathOperator{\C}{C}
\DeclareMathOperator{\W}{W}

\usepackage[colorlinks=true,
    linkcolor=black,
    citecolor=RoyalBlue,
    urlcolor=cyan, hypertexnames=false]{hyperref}
\usepackage[capitalise]{cleveref}

\title{Bach-flat metrics and global regularity of \\ conformal structures}

\date{}

\author{Rodrigo Avalos} 
\address{\noindent Institut für Mathematik, Universität Rostock, Ulmenstraße 69, Haus 3, 18057 Rostock, Germany}
\email{rodrigo.avalos@uni-rostock.de}

\author{Albachiara Cogo}
\address{\noindent Centro di Ricerca Matematica Ennio De Giorgi, Scuola Normale Superiore di Pisa, Piazza dei Cavalieri 3, 56126, Pisa, Italy}
\email{albachiara.cogo@sns.it}

\author{Andoni Royo Abrego}
\address{\noindent Institut Élie Cartan de Lorraine, Université de Lorraine, Boulevard des Aiguillettes
54506, Nancy }
\email{andoni.royo-abrego@univ-lorraine.fr}

\begin{document}

\maketitle

\begin{center}
    {\footnotesize \MakeUppercase{Rodrigo Avalos, Albachiara Cogo, and Andoni Royo Abrego}}
\end{center}

\begin{abstract}
    We study global regularity properties of conformal structures on compact manifolds with applications to conformally invariant variational problems. We find sharp conditions  ensuring that a manifold equipped with a Riemannian metric of Sobolev class is conformally diffeomorphic to a smooth one. In particular, we prove that $W^{2,q}$ Bach-flat four-manifolds with $q > 2$ are smooth up to conformal diffeomorphisms. We also solve the Yamabe problem for $W^{2,q}$ metrics with $q > 2$ on three-manifolds under a mild conformal condition.
\end{abstract}

\vspace{0.3cm}

\tableofcontents

\hfill

\pagestyle{plain}

\section[{\textbf{Introduction and main results}}]{Introduction and main results}\label{section: intro}

Arguably the most natural conformally invariant functional for a Riemannian metric in a four dimensional manifold $M$ is the \textit{Weyl energy}
\begin{equation}
    \label{eq: Weyl energy}
    \mathbf W(g) = \int_M |\W_g|^2_g \, d\mu_g \,,
\end{equation}
where $\W_g$ and $d\mu_g$ are the Weyl tensor and volume form of $g$ respectively. Critical points of \eqref{eq: Weyl energy} are metrics with vanishing Bach tensor
\begin{equation*}
    \B_{ij} = \left(\nabla^k\nabla^l + \frac{1}{2}\R^{kl}\right) \W_{ikjl} \,.
\end{equation*}
These metrics are called \textit{Bach-flat} and relatively little is known about them. Locally conformally flat, conformally Einstein and self-dual or anti-self-dual metrics are prominent examples of Bach flat metrics (we refer the reader to \cite{ViaclovskySurvey,BesseBook} for definitions and many examples). Unlike for these special cases, though, no topological obstruction for the existence of Bach flat metrics is known. Examples of Bach-flat metrics that do not belong to any of the families above have been constructed by Abbena--Garbiero--Salamon \cite{AbbenaGarbieroSalamon}.

In the influential work of Tian--Viaclovsky \cite{TianViaclovsky1,TianViaclovsky2,TianViaclovsky3}, the moduli space of smooth Bach-flat metrics was studied in analogy with previous work on the moduli space of Einstein metrics \cite{Anderson89,BandoKasueNakajima}. In particular, $\varepsilon$-regularity estimates and orbifold precompactness were established under suitable energy and non-collapsing bounds; the removable singularity theorem was completed by Streets \cite{StreetsBachflat}. Recently, a removable singularity theorem for metrics with Bach tensor bounded in $L^p$ has been proven by Laurain--Martino \cite{LaurainMartino}, with the goal of understanding bubbling along minimizing sequences rather than sequences of critical points. Malchiodi--Malizia \cite{MalchiodiMalizia,MalchiodiMalizia2} have also been able to rule out certain non self-dual/anti-self-dual bubbles along minimizing sequences by showing that the Weyl energy drops along connected sums. 

In the same direction, we are interested in understanding minimizing sequences and critical points of the Weyl energy. One of the main difficulties in the variational study of \eqref{eq: Weyl energy} is its invariance under conformal diffeomorphisms: for any positive function $u : M \to \nR$ and diffeomorphism $\Phi : M \to M$, there holds $\mathbf W(\Phi^*(u^2g)) = \mathbf W(g)$. This means that the analysis of critical points requires fixing both a diffeomorphism and a conformal gauge. Observe that the functional \eqref{eq: Weyl energy} is bounded in the space of $W^{2,q}(M)$ metrics for any $q > 2$, making it a natural candidate of admissible metrics for global analysis. In this paper, we establish that critical points to the Weyl energy in such space are smooth up to global conformal diffeomorphisms:
\begin{theorem}
    \label{thm: main}
    Let $M$ be a smooth compact four-manifold and suppose that a Riemannian metric $g \in W^{2,q}(M)$ with $q > 2$ is a critical point of
    \begin{equation*}
        \mathbf W(g) = \int_M |\W_g|^2_g \, d\mu_g \,.
    \end{equation*}
    Then, there exists a positive function $u \in W^{2,q}(M)$ and a $W^{3,q}$ diffeomorphism $\Phi : M \to M$ such that $\Phi^*(u^2g) \in C^{\infty}(M)$.
\end{theorem}
We highlight that the conformal diffeomorphism is global. If the $g$ is not a critical point of the Weyl energy, but we still have some control on the Bach tensor (this could be the case along certain Palais--Smale sequences), we are still able to find a conformal diffeomorphism that makes the metric more regular.
\begin{theorem}
    \label{thm: main Bach}
    Let $M$ be a smooth compact four-manifold and consider a Riemannian metric $g \in W^{k,q}(M)$ with $k \geq 2$ and $q > 2$. Suppose the Bach tensor satisfies $\B_g \in W^{l,q}(M)$ for some integer $l \leq k$. Then, there exists a positive function $u \in W^{k,q}(M)$ and a $W^{k+1,q}$ diffeomorphism $\Phi : M \to M$ such that $\Phi^*(u^2g) \in W^{l+4,q}(M)$.
\end{theorem}
To give an abstract viewpoint of \cref{thm: main Bach}, let $M$ more generally be a smooth compact manifold of dimension $n \geq 3$ and consider a Riemannian metric $g \in W^{k,q}(M)$ with $k \geq 2$ and $q > \tfrac{n}{2}$. Define the class
\begin{equation*}
    \llbracket\,g\,\rrbracket_{W^{k,q}} \coloneqq \Bigl\{\Phi^*\bigl(u^{\frac{4}{n-2}}g\bigr) \; : \; u \in W^{k,q}(M) , \; u > 0, \; \Phi \in \mathscr D^{k+1,q}(M)\Bigr\} \,,
\end{equation*}
where $\mathscr D^{k+1,q}(M)$ denotes the space of all $W^{k+1,q}$ diffeomorphisms from $M$ to $M$. If $g$ is a smooth Riemannian metric, it is easy to ruin its regularity simply by pulling it back via a rough diffeomorphism or by multiplying it with a rough conformal factor. In contrast, given a metric $g \in W^{k,q}(M)$, in principle there is no reason why $\llbracket\,g\,\rrbracket_{W^{k,q}}$ should contain a smooth representative. \cref{thm: main Bach} asserts that on a four-manifold, $\llbracket\,g\,\rrbracket_{W^{k,q}}$ has indeed a more regular representative provided that the regularity of the Bach tensor is better than expected. On the other hand, in view of the conformal covariance of the Bach tensor in dimension four and \cref{lemma_Adams} below, the regularity of the Bach tensor is an invariant of the class $\llbracket\,g\,\rrbracket_{W^{k,q}}$. It follows that \cref{thm: main Bach} characterizes the most regular element of the class in terms of the regularity of the Bach tensor.

Indeed, the techniques developed in this paper apply to higher dimensions and yield the following result analog to \cref{thm: main Bach} for dimensions $n \geq 4$:
\begin{theorem}
    \label{thm: main Weyl}
    Let $M$ be a smooth compact manifold of dimension $n \geq 4$ and consider a Riemannian metric $g \in W^{k,q}(M)$ with $k \geq 2$ and $q > \tfrac{n}{2}$. Suppose the Weyl tensor satisfies $\W_g \in W^{l,q}(M)$ for some integer $l \leq k$. Then, there exists a positive function $u \in W^{k,q}(M)$ and a $W^{k+1,q}$ diffeomorphism $\Phi : M \to M$ such that $\Phi^*(u^2g) \in W^{l+2,q}(M)$.
\end{theorem}
We remark that the $l \leq k$ hypothesis stops us from proving that any compact Riemannian manifold with smooth Weyl tensor is conformally diffeomorphic to a smooth one. Such obstruction is also present in \cref{thm: main Bach} and it seems unavoidable (cf. \cite[Theorem 4.5]{DeTurck_Kazdan}). Nevertheless, similar to \cref{thm: main}, we are able to overcome this obstruction if the Weyl tensor vanishes:
\begin{theorem}
    \label{thm: main Weyl-flat}
    Let $M$ be a smooth compact manifold of dimension $n \geq 4$ and consider a Riemannian metric $g \in W^{2,q}(M)$ with $q > \tfrac{n}{2}$. Suppose that $\W_g \equiv 0$. Then, there exists a positive function $u \in W^{2,q}(M)$ and a $W^{3,q}$ diffeomorphism $\Phi : M \to M$ such that $\Phi^*(u^2g) \in C^{\infty}(M)$.
\end{theorem}

An analogue result was proven by the authors \cite{AvalosCogoRoyo2} in dimension $n = 3$ for metrics $g \in W^{k,q}(M)$ with $k \geq 2$ and $q > 3$, where the conformally invariant Cotton tensor
\vspace{-0.3cm}
\begin{equation*}
    \label{eq: Cotton def}
    \C_{ijk} = \nabla_k  \R_{ij} - \nabla_j \R_{ik} + \frac{1}{4} \bigl( \nabla_j \R g_{ik} - \nabla_k \R g_{ij}\bigr)
\end{equation*}
plays the role of the Bach tensor in \cref{thm: main Bach}. In fact, the techniques developed in the present paper enable to improve our previous result:
\begin{theorem}
    \label{thm: main Cotton}
    Let $M$ be a smooth compact three-manifold and consider a Riemannian metric $g \in W^{k,q}(M)$ with $k \geq 2$ and $q > 2$. Suppose that the Cotton tensor satisfies $\C_g \in W^{l,q}(M)$ for some integer $l \leq k$. Then, there exists a positive function $u \in W^{k,q}(M)$ and a $W^{k+1,q}$ diffeomorphism $\Phi : M \to M$ such that $\Phi^*(u^2g) \in W^{l+3,q}(M)$. 
\end{theorem}
Compared with \cite[Theorem A]{AvalosCogoRoyo2}, here we work with $q > 2$ and without an orientability assumption. \cref{thm: main Cotton} can be combined with our work in \cite{AvalosCogoRoyo1}  to establish the resolution of the Yamabe problem for $W^{2,q}(M)$ metrics with $q > 2$, provided that the Cotton tensor is in $L^q(M)$. To the best of our knowledge, this is the first resolution of the Yamabe problem for a class of metrics that are not necessarily $C^{1,\alpha}(M)$ and that does not require non-positive Yamabe invariant.
\begin{corollary}
    \label{cor: Yamabe with Cotton}
    Let $M$ be an orientable smooth compact three-manifold and suppose that $g \in W^{2,q}(M)$ is a Riemannian metric with Cotton tensor $\C_g \in L^{q}(M)$ for $q > 2$. Then, there exists a metric in $\llbracket\,g\,\rrbracket_{W^{2,q}}$ of constant scalar curvature.
\end{corollary}

The main idea underlying the proofs stems from an elliptic PDE for the Ricci tensor, where the right-hand side is dominated by a conformally covariant term and a scalar-curvature term. For instance, in four dimensions we have an equation of the form
\vspace{-0.25cm}
\begin{equation*}
   \Delta \Ric = \B + \Rm * \Ric + \nabla^2 \R \,,
\end{equation*}
and similarly in other dimensions with $\B$ replaced by $\Div\C$ or $\Div\Div\W$. The notation $A*B$ for tensors $A$ and $B$ stands for some linear combination of $g$-contractions of the tensor product $A \otimes B$. Notice that the Bach tensor is conformally covariant, while $\Rm*\Ric$ is lower order. The idea then is to conformally deform the metric in a way such that $\nabla^2\R$ has better regularity. 

In \cite{AvalosCogoRoyo2}, we used our resolution of the Yamabe problem \cite{AvalosCogoRoyo1} in dimension $n = 3$ in order to bring the scalar curvature to a constant and get rid of the term $\nabla^2\R$. Nevertheless, the Yamabe problem turns out to be much harder in higher dimensions if the metric is rough, and the same idea can not be applied at the moment. In this paper, we use a first eigenfunction of the conformal Laplacian as a conformal factor, and obtain a metric whose scalar curvature is more regular, instead of constant. This is sufficient in order to gain regularity of the Ricci curvature, and to run a bootstrap argument when combined with the equation
\begin{equation*}
    g^{pq} \partial_p\partial_qg_{ij} = -2\R_{ij} + Q_{ij}\bigl(g, \partial g\bigr)
\end{equation*}
satisfied by the metric components in harmonic coordinates, where $Q_{ij}$ is quadratic in $\partial g$. To establish our global results, an essential step will be to extend some fundamental properties of $C^k$ differential structures to differential structures of Sobolev class. In this regard, we additionally give a modern proof of a classical theorem of H. Whitney, which could be of independent interest,  
see \cref{section: Sobolev version a theorem of Whitney}.

\vspace{-0.1cm}
\subsection{Outline}
The remainder of the paper is organized as follows. In \cref{section: Preliminaries}, we fix our notation and collect some geometric and analytical preliminaries needed in the later sections. In \cref{section: Sobolev version a theorem of Whitney} we prove the aforementioned technical result on differential structures of Sobolev class, and use it in \cref{section: Regularity of metric structures} in order to study the most regular global representative of an isometry class. Finally, \cref{section: Regularity of conformal structures} is dedicated to the proofs of the main theorems.

\vspace{-0.1cm}
\subsection*{Acknowledgments} The authors thank David Maxwell for interesting discussions related to this work. A. Cogo and A. Royo Abrego are grateful to the \emph{University of T\"ubingen} and the \textit{Centro di Ricerca Matematica Ennio De Giorgi}, respectively, for their hospitality during the realization of part of this work.

\section[{\textbf{Preliminaries}}]{Preliminaries}
\label{section: Preliminaries}

We will make extensive use of Sobolev spaces $W^{k,p}$ and $W^{-k,p}$ on compact manifolds for $k \in \nN_0$ and $1 < p < \infty$. We refer the reader to \cite{avalos2024sobolev} or \cite[Section 2.1]{AvalosCogoRoyo1} for precise definitions and properties.

\subsection{Riemannian curvature for metrics of Sobolev class}

We adopt the convention
\begin{align*}
    \Rm(X, Y, Z, W) &= g\bigl(\nabla_X \nabla_Y W - \nabla_Y\nabla_X W - \nabla_{[X, Y]} W, Z\bigr)
    \\
    \Ric(X,Y) &= \sum_{i=1}^n\Rm(X,e_i,Y,e_i)
    \\
    \R &= \sum_{i=1}^n\Ric(e_i,e_i)
\end{align*}
for the Riemann, Ricci, and scalar curvatures respectively, where
$X,Y,Z,W$ are arbitrary vector fields and $\{e_1, \ldots, e_n\}$ an arbitrary local orthonormal frame. This amounts to the index expressions
\begin{equation*}
    \R_{ijkl} = g_{pk}(\partial_i \Gamma^p_{jl} - \partial_j \Gamma^p_{il} + \Gamma^p_{iq}\Gamma^q_{jl} - \Gamma^p_{jq}\Gamma^q_{il}) \,, \quad \R_{ij} = g^{pq}\R_{ipjq} \,, \quad \R = g^{pq}\R_{pq} \,.
\end{equation*}
The Weyl tensor
\begin{align*}
    \W_{ijkl} &= \R_{ijkl} -\frac{1}{n-2}\left(g_{ik}\R_{jl} - g_{il}\R_{jk} - g_{jk}\R_{il} + g_{jl}\R_{ik}\right) 
    \\
    &\quad + \frac{\R}{(n-1)(n-2)} \left(g_{ik}g_{jl} - g_{il}g_{jk}\right)
\end{align*}
is the traceless part of $\R_{ijkl}$, which vanishes in dimension three or less and satisfies $\W_{u^{\frac{4}{n-2}}g} = u^{\frac{4}{n-2}}\W_g$. In dimension $n \geq 4$, we define the Bach tensor
\begin{equation}\label{def: Bach}
    \B_{ij} = \frac{1}{n-3}\nabla^p\nabla^q  \W_{qipj} + \frac{1}{n-2} \R^{pq}\W_{qipj} \,,
\end{equation}
and remark that the conformal covariance $\B_{u^2 g} = u^2 \B_g$ holds in dimension $n = 4$. Finally, we define the Cotton tensor
\begin{equation*}
    \C_{ijk} = \frac{1}{n-2}\bigl(\nabla_k  \R_{ij} - \nabla_j \R_{ik}\bigr) + \frac{1}{2(n-1)} \bigl( \nabla_j \R g_{ik} - \nabla_k \R g_{ij}\bigr)\,,
\end{equation*}
which 
is conformally invariant in dimension $n = 3$. We remark that $\B$ and $\C$ are fourth and third order in $g$ respectively, and that
\begin{equation}
    \label{eq: Cotton-Weyl}
    \nabla^p\W_{ijkp} = (n-3) \, \C_{kij} \,.
\end{equation}

\begin{lemma}
    \label{lemma: curvature continuity}
    Let $M$ be a smooth manifold of dimension $n \geq 3$ and consider a Riemannian metric $g \in W^{2,q}(M)$ for $q > \max\{2,\tfrac{n}{2}\}$. Then, the maps
    \begin{align*}
         &W^{2,q}(M) \to L^{q}(M) \,, \qquad \;\;\quad  g \mapsto \Rm_g 
         \\
         &W^{2,q}(M) \to W^{-1,q}(M) \,, \quad \quad  g \mapsto \C_g
         \\
         &W^{2,q}(M) \to W^{-2,q}(M) \,, \quad \quad g \mapsto \B_g 
    \end{align*}
    are continuous. Moreover, the Bianchi and Ricci identities
    \begin{equation*}
        \nabla_i\R_{jklm} + \nabla_k\R_{ijlm} + \nabla_j\R_{kilm} = 0
    \end{equation*}
    \begin{equation*}
        \nabla_i\nabla_j \mathbf{u}^{k}_{\;\;l} - \nabla_j\nabla_i\mathbf{u}^{k}_{\;\;l} = \R_{ij\;\;p}^{\;\;\;k} \mathbf{u}^{p}_{\;\;l} - \R_{ij\;\;l}^{\;\;\;p} \mathbf{u}^{k}_{\;\;p}
    \end{equation*}
    hold for any $(1,1)$-tensor field $\mathbf{u} \in L^q(M)$.
\end{lemma}
\begin{proof}
    Note that due to $q > \tfrac{n}{2}$, we have that $g \in L^\infty(M)$ and $\Gamma \in W^{1,q}(M)$. The first claim then follows from the form
    \begin{equation*}
        \Rm = \partial\Gamma + \Gamma * \Gamma
    \end{equation*}
    of the Riemannian curvature and the Sobolev multiplication property $W_{loc}^{1,q} \otimes W_{loc}^{1,q} \hookrightarrow L_{loc}^q$. Now, note that due to the Sobolev multiplications $W_{loc}^{1,q} \otimes L_{loc}^q \hookrightarrow W_{loc}^{-1,q}$ and $W_{loc}^{1,q} \otimes W_{loc}^{-1,q} \hookrightarrow W_{loc}^{-2,q}$ for $q > \max\{2,\tfrac{n}{2}\}$, the covariant derivative
    \begin{equation*}
        \nabla \mathbf{u} = \partial \mathbf{u} + \Gamma * \mathbf{u}
    \end{equation*}
    extends to continuous maps
    \begin{equation}
        \label{eq: nabla continuous}
        \nabla : L^q(M) \to W^{-1,q}(M) \,, \qquad \nabla : W^{-1,q}(M) \to W^{-2,q}(M) \,.
    \end{equation}
    In view of the definition of the Cotton and Bach tensors, as well as the Bianchi and Ricci identities, the rest of the assertions follow. 
\end{proof}

\subsection{The Laplacian for metrics of Sobolev class}
\label{The Laplace--Beltrami operator with Sobolev coefficients}

Let $M$ be a smooth compact manifold of dimension $n \geq 3$. Given a Riemannian metric $g \in W^{k,q}(M)$ with $k \geq 2$ and $q > \tfrac{n}{2}$, we consider its Laplacian $\Delta_g \coloneq \Div_g\nabla = \operatorname{tr}_g \nabla^2$ acting on arbitrary tensor fields. On scalar functions it is locally given by
\begin{equation*}
    \Delta_gu = \frac{1}{\sqrt{\mathrm{det}\,g}}\,\partial_{i}\Bigl(g^{ij}\sqrt{\mathrm{det}\,g}\,\partial_{j}u\Bigr) = g^{ij}\Bigl(\partial_i\partial_ju - \Gamma^k_{ij} \partial_ku\Bigr) \,.
\end{equation*}
Using the Sobolev multiplication properties and smooth approximations, it is not hard to see that it extends from $\Delta_g : C^{\infty}(M) \to W^{k-1,q}(M)$ to a bounded linear operator
\begin{equation*}
    \Delta_g : W^{k,p}(M) \to W^{k-2,p}(M)
\end{equation*}
for all $1 < p \leq q$, which is Fredholm of index zero and satisfies
\begin{equation*}
    \int_M \langle\Delta_g\mathbf{u}, \mathbf{v} \rangle_g \, d\mu_g = \int_M \langle \mathbf{u}, \Delta_g\mathbf{v} \rangle_g \, d\mu_g
\end{equation*}
for every tensor field $\mathbf{u} \in W^{2,p}(M)$ and $\mathbf{v} \in W^{2,p'}(M)$ of same rank. Indeed, it extends further by duality to a bounded linear operator
\begin{equation*}
    \Delta_g : L^p(M) \to W^{-2,p}(M)
\end{equation*}
for all $p \geq q'$, see \cite[Section 3.1]{AvalosCogoRoyo1}.

We will use the following regularity properties for the Laplacian of metrics of Sobolev class, applied to any tensor field $\mathbf{u}$. These are special cases of \cite[Proposition 5.4]{TaylorBochner} or \cite[Theorem 2.21]{HMT2023}:
\begin{theorem}
    [\textbf{Interior regularity}]
    \label{thm: Laplace interior elliptic regularity}
    Let $\Omega \subset \nR^n$ with $n \geq 3$ be an open bounded domain with smooth boundary and consider a Riemannian metric $g \in W^{k,q}_{loc}(\Omega)$ with $k \geq 2$ and $q > \tfrac{n}{2}$. 
    \begin{enumerate}
        \item[$(i)$] If $\mathbf{u} \in L^{q'}_{loc}(\Omega)$ and $\Delta_g u \in W^{-1,p}_{loc}(\Omega)$ for some $\tfrac{1}{q} - \tfrac{k-1}{n} \leq \tfrac{1}{p} \leq \tfrac{1}{q'} + \tfrac{k-1}{n}$, then $\mathbf{u} \in W^{1,p}_{loc}(\Omega)$.

        \item[$(ii)$] If $\mathbf{u} \in L^{q'}_{loc}(\Omega)$ and $\Delta_g \mathbf{u} \in W^{k-2,p}_{loc}(\Omega)$ for some $1 < p \leq q$, then $\mathbf{u} \in W^{k,p}_{loc}(\Omega)$.
    \end{enumerate}
\end{theorem}
Using a standard patching argument, one obtains the following. 
\begin{theorem}
    [\textbf{Global regularity}]
    \label{thm: Laplace global elliptic regularity}
    Let $M$ be a smooth compact manifold of dimension $n \geq 3$ and consider a Riemannian metric $g \in W^{k,q}(M)$ with $k \geq 2$ and $q > \tfrac{n}{2}$.
    \begin{itemize}
        \item[$(i)$] If $\mathbf{u} \in L^{q'}(M)$ and $\Delta_g \mathbf{u} \in W^{-1,p}(M)$ for some $\tfrac{1}{q} - \tfrac{k-1}{n} \leq \tfrac{1}{p} \leq \tfrac{1}{q'} + \tfrac{k-1}{n}$, then $\mathbf{u} \in W^{1,p}(M)$.

        \item[$(ii)$] If $\mathbf{u} \in L^{q'}(M)$ and $\Delta_g \mathbf{u} \in W^{k-2,p}(M)$ for some $1 < p \leq q$, then $\mathbf{u} \in W^{k,p}(M)$.
    \end{itemize}
\end{theorem}

\subsection{Diffeomorphisms of Sobolev class}
\label{Diffeomorphisms of Sobolev class}
In order to relate the regularity of geometric objects with respect to different coordinate systems which are not necessarily $C^\infty$ compatible, we will need the following result from \cite[Lemma A.2 and Lemma A.3]{AvalosCogoRoyo1}. This is an extension of \cite[Theorem 3.41]{AdamsFournierBook}. We adopt the convention that $\phi$ is $W^{k,q}$ diffeomorphism if $\phi$ is a $C^1$ diffeomorphism and $\phi, \phi^{-1}$ are in $W^{k,q}_{loc}$.
\begin{lemma}
    \label{lemma_Adams}
    Fix integer $k \geq 2$ and real numbers $q > \tfrac{n}{2}$ and $1 \leq p \leq q$. Let $\phi: \Omega \to \Omega'$ be a $W^{k+1,q}$ diffeomorphism between bounded open sets in $\nR^n$.
    \begin{enumerate}
        \item [$(i)$] If $u \in W^{l,p}_{loc}(\Omega)$ with $l \leq k+1$ is a scalar field, then $u \circ \phi^{-1} \in W^{l,p}_{loc}(\Omega')$ and for every smooth open set $K \subset\subset \Omega$, there exists a constant $C=C(n,p,q,k,l,K,\phi)$ such that
        \begin{equation*}
            \|u \circ \phi^{-1}\|_{W^{l,p}(\phi(K))} \leq C\|u\|_{W^{l,p}(K)} \,.
        \end{equation*}
         
        \item [$(ii)$] If $\mathbf{u} \in W^{l,p}_{loc}(\Omega)$ with $l \leq 
        k$ is a tensor field, then $\phi_*\mathbf{u} \in W^{l,p}_{loc}(\Omega')$ and for every smooth open set $K \subset\subset \Omega$, there exists a constant $C=C(n,p,q,k,l,K,\phi)$ such that
        \begin{equation*}
            \|\phi_*\mathbf{u}\|_{W^{l,p}(\phi(K))} \leq C\|\mathbf{u}\|_{W^{l,p}(K)} \,.
        \end{equation*}
     \end{enumerate}
\end{lemma}

\subsection{Harmonic coordinates and atlases}
\label{Harmonic coordinates and atlases}

The use of harmonic coordinates is going to be fundamental in this work. A coordinate system $\{x_i\}_{i=1}^n$ is called \emph{harmonic} if
\begin{equation}
    \label{conditions: harmonic_coords}
    \Delta_g x_i = 0 \qquad \text{for all } \; i = 1, \ldots, n \,.
\end{equation} 
For metrics of rough regularity, harmonic coordinates can be defined as long as \eqref{conditions: harmonic_coords} holds in a weak sense. In particular, they exists for metrics that are not necessarily $C^1$. A subtle drawback of harmonic coordinates, usually overlooked, is that their regularity depends on the regularity of the metric. This means that harmonic coordinates constructed using a non-smooth Riemannian metric are not smoothly related to the original differential structure of $M$. It will be important for us to keep track of such compatibility.
\begin{proposition}[\textbf{Harmonic coordinates}]
    \label{prop: Existence and Regularity of harmonic coordinates}
    Let $M$ be a smooth manifold of dimension $n\geq 3$ and consider a Riemannian metric $g \in W^{k,q}(M)$ with $k \geq 2$ and $q>\frac{n}{2}$. Then, there exists around each $p \in M$ a coordinate chart $(U,\varphi)$ with $\varphi(p) = 0$ satisfying
    \begin{enumerate}
        \item [$(i)$] $g_{ij}(0) = \delta_{ij}$
        \item [$(ii)$] $\Delta_g \varphi = 0$ in $U$
        \item[$(iii)$] $\varphi \in W^{k+1,q}(U,\nR^n)$
    \end{enumerate}
\end{proposition}
\begin{proof}
    See \cite[Proposition 3.11]{AvalosCogoRoyo1}.
\end{proof}
It is a simple consequence of 
\eqref{conditions: harmonic_coords} that $g^{ij}\Gamma^k_{ij} = 0$ holds in harmonic coordinates. In particular, the harmonic equation written in a harmonic coordinate chart only has top order coefficients. This implies that the compatibility between two harmonic charts is more regular than the compatibility with the original coordinates. The following result is an extension of \cite[Proposition 2.4]{AvalosCogoRoyo2} from $q > n$ to the case $q > \tfrac{n}{2}$.
\begin{proposition}[\textbf{Harmonic atlases}]
    \label{prop: harmonic atlas} 
    Let $M$ be a smooth compact manifold of dimension $n \geq 3$ and consider a Riemannian metric $g \in W^{k,q}(M)$ for $k \geq 2$ and $q > \tfrac{n}{2}$. Then, there exists a finite collection of charts $\mathcal A_H = \{(U_\alpha, \varphi_\alpha)\}_{\alpha \in \Lambda}$ covering $M$ with the following properties.
    \begin{enumerate}
        \item [$(i)$] Each chart satisfies $(i)$, $(ii)$ and $(iii)$ in \cref{prop: Existence and Regularity of harmonic coordinates}.

        \item[$(ii)$] $\mathcal A_H$ forms a $W^{k+2,q}$ atlas of $M$, i.e. 
        \begin{equation*}
            \varphi_\alpha \circ \varphi_\beta^{-1} \in W^{k+2,q}_{loc}
        \end{equation*}
        for each $\alpha,\beta \in \Lambda$. Moreover, if the metric components are in $W^{l,q}_{loc}$ in each chart of $\mathcal A_H$ for some $l \geq k$, then $\mathcal A_H$ is a $W^{l+2,q}$ atlas.
    \end{enumerate}
\end{proposition}
\begin{proof}
     By \cref{prop: Existence and Regularity of harmonic coordinates}, there exists a a harmonic coordinate chart around each point in $M$. Since $M$ is compact, it can be covered by finitely many of them $\mathcal A_H = \{(U_\alpha, \varphi_\alpha)\}_{\alpha \in \Lambda}$, which proves part $(i)$. 

    In order to prove $(ii)$, we need to examine the regularity of the transition maps between any two harmonic charts. Fix two arbitrary charts $(U_1, \varphi_1)$ and $(U_2, \varphi_2)$ in $\mathcal A_H$ with nonempty intersection and let $\{x_i\}_{i=1}^n$ and $\{y_i\}_{i=1}^n$ be the coordinates induced by $\varphi_1$ and $\varphi_2$, respectively.  
    By part $(iii)$ of \cref{prop: Existence and Regularity of harmonic coordinates} and part $(i)$ of \cref{lemma_Adams}, we directly have that $x_i(y) \in W^{k+1,q}_{loc}(\varphi_2(V_1 \cap V_2))$.
    Writing the harmonic equation satisfied by $\{x_i\}_{i=1}^n$ in terms of the harmonic coordinates $\{y_i\}_{i=1}^n$, it reads
    \begin{equation}
        \label{eq: Harmoniceq}
        \Delta_g x_i(y) = g^{pq}(y) \frac{\partial^2x_i}{\partial y_p \partial y_q}(y) = 0 \qquad \text{in} \quad \varphi_2(V_1 \cap V_2)
    \end{equation}
    for each $i = 1, \ldots, n$.  Differentiating twice in directions $y_r$ and $y_s$, we obtain
    \begin{equation}
        \label{eq: DerivativeHarmonic}
        \begin{split}
            g^{pq} \frac{\partial^2}{\partial y_p \partial y_q}&\left(\frac{\partial^2 x_i}{\partial y_r\partial y_s}\right) 
            \\
            &= - \frac{\partial^2 g^{pq}}{\partial y_r\partial y_s} \frac{\partial^2 x_i}{\partial y_p \partial y_q} - \frac{\partial g^{pq}}{\partial y_r} \frac{\partial^3 x_i}{\partial y_p \partial y_q \partial y_s} - \frac{\partial g^{pq}}{\partial y_s} \frac{\partial^3 x_i}{\partial y_p \partial y_q \partial y_r} \,.
        \end{split}
    \end{equation}
    Observe that due to $q > \tfrac{n}{2}$, we have that
    \begin{equation*}
        \frac{\partial^2 x_i}{\partial y_r\partial y_s} \in W^{k-1,q}_{loc} \hookrightarrow L^{q'}_{loc} \,,
    \end{equation*}
    whereas the right-hand side of \eqref{eq: DerivativeHarmonic} is in $W^{k-2,q}_{loc}(\varphi_2(V_1 \cap V_2)) \otimes W^{k-1,q}_{loc}(\varphi_2(V_1 \cap V_2))$. If $k \geq 3$, then
    \begin{align}
        \label{embedding k-2,q}
        W^{k-2,q}_{loc} \otimes W^{k-1,q}_{loc} &\hookrightarrow W^{k-2,q}_{loc}
    \end{align}
    holds for every $q > \tfrac{n}{2}$ and part $(ii)$ of \cref{thm: Laplace interior elliptic regularity} gives that $\partial^2x_i(y) \in W^{k,q}_{loc}(\varphi_2(V_1 \cap V_2))$, which in turn, implies that $x_i(y) \in W^{k+2,q}_{loc}(\varphi_2(V_1 \cap V_2))$. If $k = 2$ and $ q = n$, we instead employ the Sobolev multiplication
    \begin{equation*}
        L^n_{loc} \otimes W^{1,n}_{loc} \hookrightarrow L^{\frac{3n}{4}}_{loc} \,,
    \end{equation*}
    to deduce, by part $(ii)$ of \cref{thm: Laplace interior elliptic regularity}, that $\partial^2x_i(y) \in W^{2,\frac{3n}{4}}_{loc}(\varphi_2(V_1 \cap V_2))$ and thus $x_i(y) \in W^{4,\frac{3n}{4}}_{loc}(\varphi_2(V_1 \cap V_2))$. This implies, by the Sobolev multiplications
    \begin{equation*}
        W^{1,n}_{loc} \otimes W^{1,\frac{3n}{4}}_{loc} \hookrightarrow L^n_{loc} \;, \quad L^{n}_{loc} \otimes W^{2,\frac{3n}{4}}_{loc} \hookrightarrow L^n_{loc}
    \end{equation*}
    that the right-hand side of \eqref{eq: DerivativeHarmonic} is in $L^q_{loc}(\varphi_2(V_1 \cap V_2))$
    and we conclude after appealing to \cref{thm: Laplace interior elliptic regularity}. If $k = 2$ and $q < n$, we can embed 
    \begin{equation*}
        W^{1,q}_{loc} \otimes L^q_{loc} \hookrightarrow L^{p_0}_{loc} \quad \text{with} \quad \frac{1}{p_0} \coloneqq \frac{2}{q} - \frac{1}{n}
    \end{equation*}
    and deduce from part $(ii)$ of \cref{thm: Laplace interior elliptic regularity} that $x_i(y) \in W^{4,p_0}_{loc}(\varphi_2(V_1 \cap V_2))$. If $q > \tfrac{2n}{3}$, then $p_0 > \tfrac{n}{2}$ and the Sobolev multiplications
    \begin{equation*}
        L^{q}_{loc} \otimes W^{2,p_0}_{loc} \hookrightarrow L^{q}_{loc} \,, \quad W^{1,q}_{loc} \otimes W^{1,p_0}_{loc} \hookrightarrow L^{q}_{loc}
    \end{equation*}
    hold, implying that the right-hand side of \eqref{eq: DerivativeHarmonic} is in $L^q_{loc}(\varphi_2(V_1 \cap V_2))$. If $q = \tfrac{2n}{3}$, that is $p_0 = \tfrac{n}{2}$, we have the embeddings
    \begin{equation*}
        L^{q}_{loc} \otimes W^{2,p_0}_{loc} \hookrightarrow L^{\frac{4q}{5}}_{loc} \,, \quad W^{1,q}_{loc} \otimes W^{1,p_0}_{loc} \hookrightarrow L^{\frac{4q}{5}}_{loc} \,,
    \end{equation*}
    implying that the right-hand side of \eqref{eq: DerivativeHarmonic} is in $L^{\frac{4q}{5}}_{loc}(\varphi_2(V_1 \cap V_2))$ and thus $x_i(y) \in W^{4,\frac{4q}{5}}_{loc}(\varphi_2(V_1 \cap V_2))$. In turn, we deduce from the Sobolev multiplications
    \begin{equation*}
        L^{q}_{loc} \otimes W^{2,\frac{4q}{5}}_{loc} \hookrightarrow L^{q}_{loc} \,, \quad W^{1,q}_{loc} \otimes W^{1,\frac{4q}{5}}_{loc} \hookrightarrow L^{q}_{loc}
    \end{equation*}
    that the right-hand side of \eqref{eq: DerivativeHarmonic} is in $L^q_{loc}(\varphi_2(V_1 \cap V_2))$ as before. Finally, if $q < \tfrac{2n}{3}$, we learn from the multiplication properties
    \begin{equation*}
        L^{q}_{loc} \otimes W^{2,p_0}_{loc} \hookrightarrow L^{p_1}_{loc} \,, \quad W^{1,q}_{loc} \otimes W^{1,p_0}_{loc} \hookrightarrow L^{p_1}_{loc} \quad \text{with} \quad \frac{1}{p_1} \coloneqq \frac{1}{p_0} + \frac{1}{q} - \frac{2}{n} 
    \end{equation*}
    that the right-hand side is in $L^{p_1}_{loc}(\varphi_2(V_1 \cap V_2))$ and appealing to \cref{thm: Laplace interior elliptic regularity} once more, we deduce that $x_i(y) \in W^{4,p_1}_{loc}(\varphi_2(V_1 \cap V_2))$. We remark that $p_0 < p_1 < q$ since $p_0 < \tfrac{n}{2} < q$. If $p_1 > \tfrac{n}{2}$, then the Sobolev multiplications
    \begin{equation*}
        L^{q}_{loc} \otimes W^{2,p_1}_{loc} \hookrightarrow L^{q}_{loc} \,, \quad W^{1,q}_{loc} \otimes W^{1,p_1}_{loc} \hookrightarrow L^{q}_{loc}
    \end{equation*}
    imply as before that the right-hand side of \eqref{eq: DerivativeHarmonic} is in $L^q_{loc}(\varphi_2(V_1 \cap V_2))$, and if not, we use
    \begin{equation*}
        L^{q}_{loc} \otimes W^{2,p_1}_{loc} \hookrightarrow L^{p_2}_{loc} \,, \quad W^{1,q}_{loc} \otimes W^{1,p_1}_{loc} \hookrightarrow L^{p_2}_{loc} \quad \text{with} \quad \frac{1}{p_2} \coloneqq \frac{1}{p_1} + \frac{1}{q} - \frac{2}{n}
    \end{equation*}
    together with part $(ii)$ of \cref{thm: Laplace interior elliptic regularity} grant $x_i(y) \in W^{4,p_2}_{loc}(\varphi_2(V_1 \cap V_2))$. Again, $p_1 < p_2 < q$ due to $p_1 < \tfrac{n}{2} < q$. We may iterate this process to construct a monotone sequence of positive numbers $p_1 < p_2 < p_3 < \ldots$ defined by
    \begin{equation*}
        \frac{1}{p_{m+1}} \coloneq \frac{1}{p_{m}} + \frac{1}{q} - \frac{2}{n} \,.
    \end{equation*}
    The sequence does not accumulate due to the hypothesis $q > \tfrac{n}{2}$, and $p_{m+1}$ is well-defined provided that $p_m < \tfrac{n}{2}$. In particular, it reaches $\tfrac{n}{2}$ in finitely many steps, and we fall into the previous cases. We conclude that the diffeomorphism $\varphi_1 \circ \varphi_2^{-1}$ is in $W^{k+2,q}_{loc}(\varphi_2(V_1 \cap V_2))$ regardless of $k$ and $q$.

    Finally, were the components $g^{pq}$ in \eqref{eq: Harmoniceq} in $W^{l,q}_{loc}(\varphi_2(V_1 \cap V_2))$ for $l > k$, iterating the above bootstrap argument enables to get $\varphi_1 \circ \varphi_2^{-1} \in W^{l+2}_{loc}(\varphi_2(V_1 \cap V_2))$.
\end{proof}

In particular, since harmonic coordinates of rough metrics are not $C^\infty$ compatible with the original differential structure of $M$, the components of smooth tensor fields may not be $C^\infty$ in harmonic coordinates. Instead, their regularity will depend on the regularity of harmonic coordinates and, thereby, the Riemannian metric. 
\begin{corollary}
    \label{prop: Adams Lemma}    
    Let $M$ be a smooth compact manifold of dimension $n \geq 3$ and consider a Riemannian metric $g \in W^{k,q}(M)$ for $k \geq 2$ and $q > \tfrac{n}{2}$. Let $\mathcal A_H$ be a harmonic atlas as in \cref{prop: harmonic atlas} and let $\mathbf{u}$ be a tensor field on $M$. Then, $\mathbf{u} \in W^{k,q}(M)$ if and only if the components of $\mathbf{u}$ are $W^{k,q}_{loc}$ in every harmonic chart of $\mathcal A_H$.
\end{corollary}
\begin{proof}
    This is an immediate consequence of the $W^{k+1,q}$ regularity of harmonic coordinates and \cref{lemma_Adams}.
\end{proof}

\vspace{0.5cm}
\section[{\textbf{Differential structures of sobolev class}}]{Differential structures of sobolev class}
\label{section: Sobolev version a theorem of Whitney}

A foundational theorem in differential geometry due to Hassler Whitney asserts that given a $C^k$ differentiable manifold $M$ with $k \geq 1$, there always exists a $C^\infty$ differential structure on $M$ which is $C^k$ compatible. Furthermore, it is unique up to $C^\infty$ diffeomorphisms. We refer the reader to \cite[Ch. 2, Theorem 2.9]{Hirsch} for a complete proof, which makes use of Whitney's weak and strong topologies. The goal of this section is to prove an analogue result for differential structures of Sobolev class. We use a more modern approach inspired by the celebrated Meyer--Serrin theorem \cite{MeyersSerrin}, which can be easily adapted to give a more direct proof of the classical theorem.

We recall that given a $W^{k,p}$ atlas $\mathcal A$ on a manifold $M$, there exists a unique maximal $W^{k,p}$ atlas $\mathcal D$ on $M$ such that $\mathcal A \subset \mathcal D$ consisting of the set of all charts in $M$ which are $W^{k,p}$ compatible with every chart in $\mathcal A$. Such maximal atlas is referred to as the $W^{k,p}$ \emph{differential structure} generated by $\mathcal A$.
\begin{theorem}
    \label{thm: Whitney Sobolev}
    Let $\mathcal A$ be a $W^{k,q}$ differential structure on an $n$-dimensional manifold $M$ with $k \geq 3$ and $q > \tfrac{n}{2}$. Then there exists a $C^{\infty}$ differential structure $\mathcal B \subset \mathcal A$. Moreover, if $\mathcal B' \subset \mathcal A$ is another $C^\infty$ differential structure on $M$, then $(M, \mathcal B)$ is $C^\infty$ diffeomorphic to $(M, \mathcal B')$.
\end{theorem}

\begin{proof}
    Consider the set
    \begin{equation*}
        \mathscr{S} \coloneqq \Bigl\{ (B, \mathcal B) \,:\, B\subset M \text{ open}, \; \mathcal B \subset \mathcal A \text{ is a }C^{\infty}\text{ differential structure on } B \Bigr\}
    \end{equation*}
    together with the partial ordering $(B_1, \mathcal B_1) \leq (B_2, \mathcal B_2)$ whenever $B_1 \subseteq B_2$ and $\mathcal B_1 \subseteq \mathcal B_2$. Given a chain $(B_1, \mathcal B_1) \leq \cdots \leq (B_l, \mathcal B_l) \leq \cdots$, it is easy to see that 
    \begin{equation*}
        \left(\bigcup_{i = 1}^{\infty} B_i, \;\bigcup_{i = 1}^{\infty} \mathcal B_i\right) \in \mathscr S
    \end{equation*}
    is an upper bound. Moreover, $\mathscr S \neq \emptyset$, as any single chart of $\mathcal A$ defines an element of $\mathscr S$. Then, by the \textit{Zorn's lemma}, there exists a maximal element $(B, \mathcal B) \in \mathscr S$. If $B = M$ we are done, so let us assume for a contradiction that $B \neq M$. Then, there exists a chart $(W, \phi) \in \mathcal A$ such that $W \cap (M \setminus B) \neq \emptyset$ and $W \cap B \neq \emptyset$. By local compactness of manifolds, we may restrict to a precompact domain $U  \subset\subset W$ still satisfying $U \cap (M \setminus B) \neq \emptyset$ and $U \cap B \neq \emptyset$ and assume that $\phi(U) = B_1(0)$. Our goal is to construct a chart $(U,\widetilde{\phi}) \in \mathcal A$ that is $C^\infty$ compatible with $\mathcal B$ and contradict the maximality of $(B,\mathcal B)$. 
    
    Let $V \coloneq U \cap B$ and consider for each $m \in \nN$ the open set
    \begin{equation*}
        V_m \coloneqq \Big\{p \in V \; : \; 2^{-m-1} < \inf_{q \in \nR^n \setminus \phi(V)} |\phi(p) - q| < 2^{-m+1}\Bigr\} \,.
    \end{equation*}
    We remark that $V_m$ only intersects $V_{m-1}$ and $V_{m+1}$ for each $m \geq 2$. Note that closure $\overline{V_m}$ is contained in $\overline U$, which is compact, and therefore $\overline{V_m}$ is itself compact. Moreover, $\overline{V_m} \subset V \subset B$. Hence, there exists a finite collection of charts of $\mathcal B$ covering $\overline{V_m}$, which after restricting to $V_m$, we name $\big\{\big(U_{m, i}\,,\,\varphi_{m, i}\big)\big\}_{i = 1}^{\Lambda_m}$. Note that since $\Lambda_m \in \nN$ is finite for each $m$, we have a locally finite covering
    \begin{equation}
        \label{eq: loc fin cover}
        V = \bigcup_{m \in \nN} \,\bigcup_{i = 1}^{\Lambda_m} U_{m, i} \,. 
    \end{equation}
    We can thereby consider a smooth  partition of unity $\big\{\eta_{m,i}\big\}_{i=1, m \in \nN}^{\Lambda_m}$ subordinate to it. Moreover, by our construction, we may assume that $\supp \eta_{m,i} \subset U_{m,i}$ is compact, see \cite[Appendix C]{TuBook}. Now, let us define the functions $f_{m,i}: \varphi_{m,i}(U_{m,i}) \to \nR^n$ by
    \begin{equation*}
        f_{m, i} \coloneqq (\eta_{m, i} \,\phi) \circ \varphi_{m, i}^{\,-1} \;,
    \end{equation*}
    which have compact support $\supp f_{m,i} \subset \varphi_{m,i}(U_{m,i})$ and are, by the regularity of $\mathcal A$, in $W^{k,q}_0(\varphi_{m,i}(U_{m,i}))$. We extend them by zero as necessary, and consider for each $m$ and $i$, a bounded smooth open set $K_{m,i}$ such that
    \begin{equation*}
        \supp f_{m,i} \subset\subset K_{m,i} \subset\subset \varphi_{m,i}(U_{m,i}) \,.
    \end{equation*}
    In particular, for any $\varepsilon > 0$ there exist $f_{m,i}^\varepsilon \in C^\infty_c(K_{m,i})$ with the property that
    \begin{equation}
        \label{eq: estimate f_epsilon}
        \|f_{m, i} - f_{m, i}^\varepsilon\|_{W^{k,q}(K_{m,i})} < \frac{\varepsilon}{2^m  C_m}
    \end{equation}
    with
    \begin{equation}
        \label{eq: Cm constants}
        C_{m} = \sum_{i=1}^{\Lambda_m} C_{m,i}
    \end{equation}
    where $C_{m,i}$ are the constants in \cref{lemma_Adams} associated to the diffeomorphisms $\phi \circ \varphi_{m,i}^{-1}$ and open sets $K_{m,i}$. Consider the mapping $\widetilde\phi : U \to \nR^n$ defined by
    \begin{equation*}
        \widetilde{\phi} \coloneqq
        \begin{dcases}
            \sum_{m \in \nN}\sum_{i =1}^{\Lambda_m} f_{m, i}^\varepsilon \circ \varphi_{m, i} \quad &\text{in } \; V
            \\
            \phi &\text{in } \; U \setminus V
        \end{dcases}
    \end{equation*}
    where $f_{m, i}^\varepsilon \circ \varphi_{m, i}$ are taken to vanish outside their domain of definition $U_{m,i} \subset V$. Note that the infinite series is (pointwise) well-defined by the local finiteness of the covering \eqref{eq: loc fin cover}. We claim that $(U, \widetilde{\phi})$ is a chart in $\mathcal A$, which is $C^\infty$ compatible with $\mathcal B$, thus reaching the desired contradiction. For the second condition, consider any chart $(W,\psi) \in \mathcal B$ intersecting $U$ and notice that 
    \begin{equation*}
        \widetilde{\phi} \circ \psi^{-1} = \sum_{m \in \nN}\sum_{i = 1}^{\Lambda_m} f_{m, i}^\varepsilon \circ \left(\varphi_{m, i}^{\,} \circ \psi^{-1}\right) \,,
    \end{equation*}
    in $U \cap W$. Since $\varphi_{m, i}^{\,} \circ \psi^{-1} \in C^\infty_{loc}$, we have that $\widetilde{\phi} \circ \psi^{-1} \in C^\infty_{loc}$ due to \cite[Theorem 3.41]{AdamsFournierBook}. In order to show that $(U,\widetilde{\phi})$ is a chart in $\mathcal A$, it suffices to verify that $\widetilde{\phi} \circ \phi^{-1} : \phi(U) \to \widetilde{\phi}(U)$ is a $W^{k,q}$ diffeomorphism. To do so, let us consider for $N \in \nN$ the mappings
    \begin{equation*}
        \phi_N =
        \begin{dcases}
            \sum_{m = 1}^N\sum_{i =1}^{\Lambda_m} \eta_{m,i} \phi \quad &\text{in } \; V
            \\
            \phi &\text{in } \; U \setminus V
        \end{dcases}
    \end{equation*}
    and
    \begin{equation*}
        \widetilde{\phi}_N \coloneqq
        \begin{dcases}
            \sum_{m = 1}^N\sum_{i =1}^{\Lambda_m} f_{m, i}^\varepsilon \circ \varphi_{m, i} \quad &\text{in } \; V
            \\
            \phi &\text{in } \; U \setminus V
        \end{dcases}
    \end{equation*}
    defined on $U$. Again, we remark that the compactly supported maps $\eta_{m,i} \phi$ and $f_{m, i}^\varepsilon \circ \varphi_{m, i}$ are to be understood to extend by zero to the whole $V$. Then, 
    \begin{equation}
        \label{eq: phi_N difference}
        \left(\widetilde{\phi}_N - \phi_N \right) \circ \phi^{-1} =
        \begin{dcases}
            \sum_{m = 1}^N \sum_{i =1}^{\Lambda_m} \left(f_{m, i}^\varepsilon \circ \varphi_{m, i} - \eta_{m, i} \phi \right) \circ \phi^{-1} \quad &\text{in } \; \phi(V)
            \\
            0 &\text{in } \; \phi(U \setminus V)
        \end{dcases}
    \end{equation}
    are well defined in $W^{k,q}_{0}(\phi(U))$ by \cref{lemma_Adams} and converge pointwise to $\widetilde{\phi} \,\circ\, \phi^{-1} - \Id$. In fact, for any $N > L \geq 1$, there holds
    \begin{align*}
        \| (\widetilde{\phi}_{N} - \phi_{N})  \circ \phi^{-1} - & (\widetilde{\phi}_{L}  - \phi_{L}) \circ \phi^{-1}\|_{W^{k,q}(\phi(U))}
        \\
        &\leq \sum_{m=L+1}^N \sum_{i = 1}^{\Lambda_m} \| (f_{m,i}^\varepsilon \circ \varphi_{m,i} - \eta_{m,i} \,\phi) \circ \phi^{-1} \|_{W^{k,q}(\phi(U))}
        \\
        &= \sum_{m=L+1}^N \sum_{i=1}^{\Lambda_m} \|(f_{m,i}^\varepsilon - (\eta_{m,i}\phi) \circ \varphi_{m,i}^{-1}) \circ (\varphi_{m,i} \circ \phi^{-1})\|_{W^{k,q}(\phi(U))}
        \\
        &= \sum_{m=L+1}^N \sum_{i=1}^{\Lambda_m} \|(f_{m,i}^\varepsilon - f_{m,i} ) \circ (\phi \circ \varphi_{m,i}^{-1})^{-1}\|_{W^{k,q}(\phi(\varphi_{m,i}^{-1}(K_{m,i})))}  
        \\
        &\leq \sum_{m=L+1}^N \sum_{i = 1}^{\Lambda_m} C_{m,i} \, \|f_{m,i}^\varepsilon - f_{m,i} \|_{W^{k,q}(K_{m,i})}
        \\
        &\leq \sum_{m=L+1}^N \sum_{i=1}^{\Lambda_m} C_{m, i} \, \frac{\varepsilon}{2^m C_m} \leq \sum_{m=L+1}^N\frac{\varepsilon}{2^m} \,,
    \end{align*}
    where in the last inequalities we have used the estimate of \cref{lemma_Adams} together with \eqref{eq: estimate f_epsilon} and \eqref{eq: Cm constants}. It follows that \eqref{eq: phi_N difference} is a Cauchy sequence in $W^{k, q}(\phi(U))$ and in particular converges to $\widetilde \phi \circ \phi^{-1} - \Id$ in $W^{k,q}(\phi(U))$. Moreover, using analogous estimates, we obtain that
    \begin{align*}
        \| \widetilde{\phi} \circ \phi^{-1} - \Id\|_{W^{k, q}(\phi(U))} & \leq \sum_{m \in \nN} \sum_{i = 1}^{\Lambda_m} \| \left(f_{m, i}^\varepsilon \circ \varphi_{m, i} - \eta_{m, i} \phi \right) \circ \phi^{-1} \|_{W^{k, q}(\phi(U))} 
        \\
        & \leq \sum_{m \in \nN} \sum_{i = 1}^{\Lambda_m} C_{m, i} \, \frac{\varepsilon}{2^m C_m} = 2\varepsilon \,. 
    \end{align*}
    Since $W^{k,q} \hookrightarrow C^1$, we have that $\widetilde{\phi} \circ \phi^{-1} - \Id$ is arbitrarily small in $C^1$, implying that $\widetilde{\phi} \circ \phi$ is a $C^1$ diffeomorphism. Since $\widetilde{\phi} \circ \phi^{-1} \in W^{k,q}(\phi(U))$, \cite[Theorem A.1]{AvalosCogoRoyo1} implies that its inverse is in $W^{k,q}_{loc}(\widetilde\phi(U))$ and consequently $\widetilde{\phi} \circ \phi^{-1}$ is a $W^{k,q}$ diffeomorphism. This contradicts the maximality of $(B,\mathcal B)$ and proves that $B = M$.

    For the second part, suppose that $\mathcal B' \subset \mathcal A$ is another $C^\infty$ differential structure on $M$. In particular, it is $W^{k,q}$ compatible to $\mathcal B$ and due to $k \geq 3$ and $q > \tfrac{n}{2}$, both $\mathcal B$ and $\mathcal B'$ are $C^\infty$ differential structures on $M$ which are $C^1$ compatible to each other. Then, it follows from the density theorem \cite[Theorem 2.7]{Hirsch} that the identity map $\Id : (M,\mathcal B) \to (M,\mathcal B')$ can be approximated by a $C^\infty$ diffeomorphism $\Phi : (M,\mathcal B) \to (M,\mathcal B')$.
\end{proof}

\vspace{0.5cm}
\section[{\textbf{Regularity of metric structures}}]{Regularity of metric structures}
\label{section: Regularity of metric structures}

With the help of \cref{thm: Whitney Sobolev}, we can prove the following theorem concerning the optimal regularity of the elements in an isometry class, which extends to $q > \tfrac{n}{2}$ our previous result \cite[Theorem C]{AvalosCogoRoyo2}. The original (local) idea goes back to Sabitov--Shefel \cite{Sabitov-Shefel} and De Turck--Kazdan \cite{DeTurck_Kazdan}.
\begin{theorem}
    \label{thm: regularity metric struictures}
    Let $M$ be a smooth compact manifold of dimension $n \geq 3$ and consider a Riemannian metric $g \in W^{k,q}(M)$ for $k \geq 2$ and $q > \frac{n}{2}$. Suppose that $\Ric_g \in W^{l,q}(M)$ for some $l \leq k$. Then, there exists a $W^{k+1,q}$ diffeomorphism $\Phi : M \to M$ such that $\Phi^*g \in W^{l+2,q}(M)$.
\end{theorem}
\begin{proof}
    If $l < k-1$, the statement is trivial, so we assume that $k-1 \leq l$. By \cref{prop: harmonic atlas}, there exists a finite atlas $\mathcal A_H$ of $M$ consisting of harmonic charts, which is $W^{k+1,q}$ compatible with the $C^\infty$ differential structure of $M$. In any such harmonic chart $(V,\varphi) \in \mathcal A_H$, with coordinates $\{y_i\}_{i=1}^n$, the metric components satisfy 
    \begin{equation}
        \label{eq: Ricci harmonic PDE}
        g^{pq} \partial_p\partial_qg_{ij} = -2\R_{ij} +  \,Q_{ij}\bigl(g, \partial g\bigr) \,,
    \end{equation}
    where  $Q_{ij}$ is quadratic in $\partial g$. 

    \medskip
    \textit{Step 1.} \, First we show that $g_{ij} \in W^{l+2,q}_{loc}(\varphi(V))$. Differentiating \eqref{eq: Ricci harmonic PDE} in direction $y_s$ and rearranging, we obtain
    \begin{equation}
        \label{eq: differentiation of harmonic ricci}
        g^{pq}\partial_p\partial_q \bigl(\partial_sg_{ij}\bigl) = -2\,\partial_s\R_{ij} + \partial_sQ_{ij} - \partial_sg^{pq} \, \partial_p\partial_qg_{ij} \,,
    \end{equation}
    where the terms on the right-hand side satisfy
    \begin{align*}
        \partial_s\R_{ij} \in W^{l-1,q}_{loc}(\varphi(V)) \,,\quad \partial_sQ_{ij} \,, \; \partial_sg^{pq} \, \partial_p\partial_qg_{ij} \in W^{k-1,q}_{loc}(\varphi(V)) \otimes W^{k-2,q}_{loc}(\varphi(V))
    \end{align*}
    due to the hypotheses and \cref{prop: Adams Lemma}. Notice that due to $q > \tfrac{n}{2}$, the Sobolev embedding $W^{k-1,q}_{loc}(\varphi(V)) \hookrightarrow L^{q'}_{loc}(\varphi(V))$ holds for $n \geq 3$, and then $\partial_sg_{ij} \in L^{q'}_{loc}(\varphi(V))$. 
    
    If $k \geq 3$, the right-hand side of \eqref{eq: differentiation of harmonic ricci} is in $W^{k-2,q}_{loc}(\varphi(V))$ and \cref{thm: Laplace interior elliptic regularity} implies that $\partial_sg_{ij}\in W^{k,q}_{loc}(\varphi(V))$. In turn, we get that $g_{ij} \in W^{k+1,q}_{loc}(\varphi(V))$. Now, if $l = k$, the right hand side of \eqref{eq: differentiation of harmonic ricci} is in $W^{k-1,q}_{loc}(\varphi(V))$ and appealing to \cref{thm: Laplace interior elliptic regularity} one more time we deduce that $g_{ij} \in W^{k+2,q}_{loc}(\varphi(V))$. Thereby, $g_{ij} \in W^{l+2,q}_{loc}(\varphi(V))$ regardless of $l$.

    If $k = 2$ and $q > n$ the same Sobolev multiplications hold and the proof is similar, as in \cite[Theorem C]{AvalosCogoRoyo2}. On the other hand, for $q = n$, the embedding
      \begin{equation*}
        \partial_sQ_{ij}\,, \; \partial_sg^{pq} \, \partial_p\partial_qg_{ij} \in W^{1,n}_{loc}(\varphi(V)) \otimes L^{n}_{loc}(\varphi(V)) \hookrightarrow  L^{\frac{3n}{4}}_{loc}(\varphi(V)) 
    \end{equation*}
    ensures that the right-hand side of \eqref{eq: differentiation of harmonic ricci} is in $L^{\frac{3n}{4}}_{loc}(\varphi(V))$ and \cref{thm: Laplace interior elliptic regularity} implies that $\partial_s g_{ij} \in W^{2, \frac{3n}{4}}_{loc}(\varphi(V))$. In turn, $g_{ij} \in W^{3,\frac{3n}{4}}_{loc}(\varphi(V)) \hookrightarrow W^{2,3n}_{loc}(\varphi(V))$ by Sobolev embedding and we are again in the case $q > n$. On the other hand, if $q < n$, we
    consider the embedding 
    \begin{equation*}
        \partial_sQ_{ij}\,, \; \partial_sg^{pq} \, \partial_p\partial_qg_{ij} \in W^{1,q}_{loc}(\varphi(V)) \otimes L^{q}_{loc}(\varphi(V)) \hookrightarrow  L^{p_0}_{loc}(\varphi(V)) \,, \qquad \frac{1}{p_0} \coloneq \frac{2}{q} - \frac{1}{n} \,.
    \end{equation*}
    It is easy to check that $\tfrac{n}{3} < p_0 < q$ due to $\tfrac{n}{2} < q < n$. 
    Hence, the right-hand side of \eqref{eq: differentiation of harmonic ricci} is in $L^{p_0}_{loc}(\varphi(V))$ and we can apply \cref{thm: Laplace interior elliptic regularity} to obtain that $g_{ij} \in W^{3,p_0}_{loc}(\varphi(V))$. 
    Now, if $p_0 \geq \tfrac{2nq}{3q+n}$, the Sobolev embedding
    \begin{equation}
        \label{embedding p_0 q}
        \partial_sQ_{ij}\,, \; \partial_sg^{pq} \, \partial_p\partial_qg_{ij} \in W^{2, p_0}_{loc}(\varphi(V)) \otimes W^{1, p_0}_{loc}(\varphi(V)) \hookrightarrow  L^{q}_{loc}(\varphi(V))
    \end{equation}
    holds and \cref{thm: Laplace interior elliptic regularity} grants $g_{ij} \in W^{3,q}_{loc}(\varphi(V))$. If $p_0 < \tfrac{2nq}{3q+n}$ instead, in particular $p_0 < \tfrac{n}{2}$ because of $q < n$ and we may consider the embedding
    \begin{equation*}
         W^{2,p_0}_{loc}(\varphi(V)) \otimes W^{1, p_0}_{loc}(\varphi(V)) \hookrightarrow  L^{p_1}_{loc}(\varphi(V)) \,, \qquad \frac{1}{p_1} \coloneq \frac{2}{p_0} - \frac{3}{n} \,.
    \end{equation*}
    One can check that $p_0 < p_1 < \infty$ due to $p_0 > \tfrac{n}{3}$. We deduce that the right-hand side of \eqref{eq: differentiation of harmonic ricci} is in $L^{p_1}_{loc}(\varphi(V))$ and thus $g_{ij} \in W^{3,p_1}_{loc}(\varphi(V))$ by \cref{thm: Laplace interior elliptic regularity}. If $p_1 \geq \tfrac{2nq}{3q+n}$, we are done in one step as before. If not, we may iterate this process for a sequence of increasing positive numbers $p_1 < p_2 < \ldots < p_M$ defined by
    \begin{equation*}
        \frac{1}{p_{m+1}} \coloneq \frac{2}{p_m} - \frac{3}{n}
    \end{equation*}
    as long as $p_m < \tfrac{2nq}{3q+n}$. In view of
    \begin{equation*}
        \frac{1}{p_m} - \frac{1}{p_{m+1}} = \frac{3}{n} - \frac{1}{p_m} > \frac{3}{n} - \frac{1}{p_0} = \frac{4}{n} - \frac{2}{q} > 0 \,,
    \end{equation*}
    we conclude that the sequence $(p_m)_m$ reaches $\tfrac{2nq}{3q+n}$ in finitely many steps and, by the arguments above, we obtain that $g_{ij} \in W^{3,q}_{loc}(\varphi(V))$. In turn, we have that
    \begin{equation*}
        \partial_sQ_{ij} \,, \; \partial_sg^{pq} \, \partial_p\partial_qg_{ij} \in W^{2,q}_{loc}(\varphi(V)) \otimes W^{1,q}_{loc}(\varphi(V)) \hookrightarrow W^{1,q}_{loc}(\varphi(V))
    \end{equation*}
    and if $l = k = 2$, the right-hand side of \eqref{eq: differentiation of harmonic ricci} is in $W^{1,q}_{loc}(\varphi(V))$. \cref{thm: Laplace interior elliptic regularity} then implies that $g_{ij} \in W^{4,q}_{loc}(\varphi(V))$. We have thus established that $g_{ij} \in W^{l+2,q}_{loc}(\varphi(V))$ regardless of $k \geq 2$, $q > \tfrac{n}{2}$ and $l \leq k$.

    \textit{Step 2.} \, To conclude the proof we use the same argument as in \cite[Theorem C]{AvalosCogoRoyo2}, but this time appealing to \cref{thm: Whitney Sobolev} rather than the original classical theorem for $C^k$ differential structures: given the local improvement of metric components in harmonic coordinates, part $(ii)$ of \cref{prop: harmonic atlas} implies that the harmonic charts in $\mathcal A_H$ are $W^{l+4}$ compatible to each other and thus generate a $W^{l+4,q}$ differential structure on $M$. Moreover, $g$ is $W^{l+2, q}$ regular with respect to such differential structure. By \cref{thm: Whitney Sobolev}, there exists a $C^\infty$ differential structure on $M$ which is $W^{l+4,q}$ compatible to $\mathcal A_H$. Denoting by $M'$ the manifold $M$ endowed with this new $C^\infty$ differential structure and using \cref{prop: Adams Lemma}, we get that $g \in W^{l+2,q}(M')$. Moreover, \cref{thm: Whitney Sobolev} implies that there exists a $C^\infty$ diffeomorphism $\Phi : M \to M'$. Consequently, $\Phi^*g \in W^{l+2,q}(M)$. Note, however, that $\Phi \in W^{k+1,q}(M,M)$, since the regularity of harmonic coordinates in the original charts of $M$ implies that these two $C^\infty$ differential structures are only $W^{k+1,q}$ compatible.
\end{proof}

\vspace{0.5cm}
\section[{\textbf{Regularity of conformal structures}}]{Regularity of conformal structures}
\label{section: Regularity of conformal structures}

In this section we prove the main theorems of the paper. In view of \cref{thm: regularity metric struictures}, the goal is to use the hypotheses on the regularity of a conformally covariant tensor in order to find a conformal deformation into a metric whose Ricci tensor has good regularity properties, and then bootstrap the argument. We begin by deriving the main geometric PDEs, different versions thereof have appeared in \cite{TianViaclovsky1,avalos2024sobolev,AvalosCogoRoyo2,LaurainMartino}.
\begin{lemma}
    \label{lemma: conformal PDEs lowreg}
    Let $M$ be a smooth manifold of dimension $n \geq 3$ and consider a Riemannian metric $g \in W^{2,q}(M)$ for $q > \max\{2,\frac{n}{2}\}$. Then, an equation of the type
    \begin{equation}
    \label{eq: Cotton}
         \Delta \Ric  = \Div\C + \Rm * \Ric + \nabla^2\R
    \end{equation}
    holds in $W^{-2,q}(M)$. In particular, for $n \geq 4$, 
    \begin{equation}
    \label{eq: Weyl}
        \Delta \Ric  = \Div\Div\W + \Rm * \Ric + \nabla^2\R
    \end{equation}
    \begin{equation}
    \label{eq: Bach}
          \Delta \Ric  = \B + \Rm * \Ric + \nabla^2\R
    \end{equation}
    hold in $W^{-2,q}(M)$.
\end{lemma}
\begin{proof}
   It was proven in \cite[Lemma 4.1]{AvalosCogoRoyo2} that
   \begin{equation}
       \begin{split}
           \Delta \R_{ij} & = (n-2) \nabla^k\C_{ijk} - \R_i^{\;l}\R_{jl} - \R_{j\,\;i}^{\;kl} \R_{kl} 
           \\
           & \qquad \qquad \qquad + \frac{n-2}{2(n-1)} \Bigl(\nabla_i \nabla_j \R + \frac{1}{n-2}  \Delta \R  \, g_{ij}\Bigr)
        \end{split}
    \end{equation}
    holds in $W^{-2,q}(M)$, which gives \eqref{eq: Cotton}. If $n \geq 4$, simply by inserting \eqref{eq: Cotton-Weyl} and \eqref{def: Bach}, which hold in $W^{-1,q}(M)$ and $W^{-2,q}(M)$ respectively by the continuity of \eqref{eq: nabla continuous}, we have that
    \begin{equation}
    \begin{split}
        \Delta \R_{ij} & = \frac{n-2}{n-3} \, \nabla^k\nabla^l \W_{iljk} - \R_i^{\;l}\R_{jl} - \R_{j\,\;i}^{\;kl} \R_{kl} \\
        & \qquad \qquad \qquad + \frac{n-2}{2(n-1)} \Bigl(\nabla_i \nabla_j \R + \frac{1}{n-2}  \Delta \R  \, g_{ij}\Bigr)
    \end{split}
    \end{equation}
    and
    \begin{equation}
    \begin{split}
          \Delta \R_{ij} & = (n-2) \B_{ij}-  \R^{pq}\W_{qipj} - \R_i^{\;l}\R_{jl} - \R_{j\,\;i}^{\;kl} \R_{kl} \\
          & \qquad \qquad \qquad + \frac{n-2}{2(n-1)} \Bigl(\nabla_i \nabla_j \R + \frac{1 }{n-2} \Delta \R \,  g_{ij}\Bigr)
    \end{split}
    \end{equation}
    hold in $W^{-2,q}(M)$. This establishes \eqref{eq: Weyl} and \eqref{eq: Bach}.
\end{proof}

Next, we introduce the following simple but powerful conformal deformation by first eigenfunctions. Such good \emph{conformal gauge} was used in \cite{Maxwell1} and \cite{avalos2024sobolev}, as well as in the resolution of the Yamabe problem in \cite{AvalosCogoRoyo1}.
\begin{lemma}
    \label{lemma: good conformal gauge}
    Let $M$ be a smooth compact manifold of dimension $n \geq 3$ and consider a Riemannian metric $g \in W^{k,q}(M)$ for $k \geq 2$ and $q > \frac{n}{2}$.
    There exists a positive function $u \in W^{k,q}(M)$ such that the conformal metric $u^{\frac{4}{n-2}} g$ has scalar curvature in $W^{k, q}(M)$.
\end{lemma}
\begin{proof}
    Using the direct method of calculus of variations, it is not hard to show the existence of a positive function $u \in W^{k,q}(M)$ satisfying the eigenvalue problem 
    \begin{equation*}
         -4 \, \frac{n-1}{n-2}\Delta_gu + \R_gu = \lambda u
    \end{equation*}
    for some $\lambda \in \nR$, see \cite[Section 2.3]{AvalosCogoRoyo1} or \cite[Section 3.3.2]{RoyoThesis} for instance. By the standard formula for conformal deformations of scalar curvature, the scalar curvature of the metric $\tilde g = u^{\frac{4}{n-2}} g$ is given by
    \begin{equation*}
        \R_{\tilde g} = u^{-\frac{n+2}{n-2}} \left( -4 \frac{n-1}{n-2}\Delta_gu + \R_gu\right) = \lambda u^{-\frac{4}{n-2}}
    \end{equation*}
    and is thus of class $W^{k,q}(M)$.
\end{proof}
We are now ready to prove \cref{thm: main Bach}:
\begin{proof}[Proof of \cref{thm: main Bach}]
    If $l < k-3$, the statement is trivially true with $\Phi = id_M$ and $u \equiv 1$, so let us assume that $k-3 \leq l$. Let $g_1 = u_1^2 g$ with $u_1 \in W^{k,q}(M)$ be the conformal deformation given in \cref{lemma: good conformal gauge}, so that $\R_{g_1} \in W^{k,q}(M)$. Consider the equation
    \begin{equation}
        \label{eq: PDE g_1}
        \Delta_{g_1}\Ric_{g_1} = \B_{g_1} + \Ric_{g_1}*\Rm_{g_1} + \nabla^2\R_{g_1}
    \end{equation}
    from \cref{lemma: conformal PDEs lowreg} and note that $\Ric_{g_1} \in L^{q'}(M)$ due to $q > 2$. In view of the conformal covariance $\B_{u^2g} = u^{-2}\B_g$
    of the Bach tensor, we have $\B_{g_1} \in W^{l,q}(M)$, while
    \begin{align*}
        \Ric_{g_1}*\Rm_{g_1} \in W^{k-2,q}(M) \otimes W^{k-2,q}(M) \hookrightarrow W^{k-3,q}(M)
    \end{align*}
    holds whenever $k \geq 3$, or $k = 2$ and $q \geq 4$. In this case, the right-hand side of \eqref{eq: PDE g_1} is in $W^{k-3,q}(M)$ and \cref{thm: Laplace global elliptic regularity} implies that $\Ric_{g_1} \in W^{k-1,q}(M)$. If $k = 2$ but $2 < q < 4$, then the Sobolev multiplication
    \begin{equation*}
        L^{q} \otimes L^{q} \hookrightarrow W^{-1,p_0} \,, \qquad \frac{1}{p_0} \coloneq \frac{2}{q} - \frac{1}{4}
    \end{equation*}
    holds. Note that $1 < p_0 < q$ and therefore the right-hand side of \eqref{eq: PDE g_1} is in $W^{-1,p_0}(M)$. Appealing to \cref{thm: Laplace global elliptic regularity}, we deduce that $\Ric_{g_1} \in W^{1,p_0}(M)$. If $p_0 \geq 2$, then the Sobolev embedding $W^{1,p_0} \otimes L^{q} \hookrightarrow W^{-1,q}$ holds and $\Ric_{g_1} \in W^{1,q}(M)$ follows from \cref{thm: Laplace global elliptic regularity}. If $p_0 < 2$, then we use
    \begin{equation*}
        W^{1,p_0} \otimes L^{q} \hookrightarrow W^{-1,p_1} \,, \qquad \frac{1}{p_1} \coloneq \frac{1}{p_0} + \frac{1}{q} - \frac{1}{2}
    \end{equation*}
    and deduce $\Ric_{g_1} \in W^{1,p_1}(M)$. Note that $p_1 > p_0$ due to $q > 2$. If $p_1 \geq 2$ then it follows that $\Ric_{g_1} \in W^{1,q}(M)$ as before, and otherwise we iterate the argument with a sequence of embeddings
    \begin{equation*}
        W^{1,p_{m}} \otimes L^{q} \hookrightarrow W^{-1,p_{m+1}} \,, \qquad \frac{1}{p_{m+1}} \coloneq \frac{1}{p_m} + \frac{1}{q} - \frac{1}{2}
    \end{equation*}
    as long as $p_m < 2$. In view of the hypothesis $q > 2$, the increasing sequence $(p_m)_m$ of real numbers does not accumulate and reaches $2$ in finitely many steps. We conclude that $\Ric_{g_1} \in W^{1,q}(M)$. We have thus established that $\Ric_{g_1} \in W^{k-1,q}(M)$ regardless of $k \geq 2$ and $q > 2$. At this point, we may appeal to \cref{thm: regularity metric struictures} to deduce the existence of a global diffeomorphism $\Phi_1 \in W^{k+1,q}(M,M)$ such that $\Phi_1^*(u^2_1g) \in W^{k+1,q}(M)$.

    If $l = k-3$ we are done. If $k-2 \leq l \leq k$, we name $\tilde g_1 = \Phi_1^*(u^2_1g)$ and notice that $\B_{\tilde g_1} = \Phi_1^*(\B_{g_1}) \in W^{l,q}(M)$ due to the regularity of $\Phi_1$ and part $(ii)$ of \cref{lemma_Adams}. We can now restart the argument: let $u_2 \in W^{k+1,q}(M)$ be the first positive eigenfunction of the conformal Laplacian associated to $\tilde g_1$ as in \cref{lemma: good conformal gauge}. Consider the conformal metric $g_2 = u^2_2 \,  \tilde g_1 \in W^{k+1,q}(M)$, whose scalar curvature is in $W^{k+1,q}(M)$ by \cref{lemma: good conformal gauge}. Now we have
    \begin{equation}
        \label{eq: PDE g_2}
        \Delta_{g_2}\Ric_{g_2} = \B_{g_2} + \Ric_{g_2}*\Rm_{g_2} + \nabla^2\R_{g_2}
    \end{equation}
    with the right-hand side in $W^{k-2,q}(M)$ thanks to the Sobolev multiplication
    \begin{align*}
        W^{k-1,q} \otimes W^{k-1,q} \hookrightarrow W^{k-2,q} \,.
    \end{align*}
    It follows from \cref{thm: Laplace global elliptic regularity} that $\Ric_{g_2} \in W^{k,q}(M)$ and, appealing to \cref{thm: regularity metric struictures}, we establish that for some diffeomorphism $\Phi_2 \in W^{k+2,q}(M,M)$, there holds
    \begin{align*}
        \Phi^*_2&(u_2^2 \, \tilde g_1) = \Phi^*(u^2 g) \in W^{k+2,q}(M)
    \end{align*}
    with $\Phi = \Phi_1 \circ \Phi_2$ and $u = (u_2 \circ \Phi_1^{-1}) u_1$. Notice that due to the regularity of $\Phi_1$ and $u_1$, we only have $\Phi \in W^{k+1,q}(M,M)$ and $u \in W^{k,q}(M)$ via \cref{lemma_Adams}.

    If $l = k-2$ we are done, and if $k-1 \leq l \leq k$, we name $\tilde g_2 = \Phi_2^*g_2$. Since $\B_{\tilde g_2} = \Phi_2^*(\B_{g_2}) \in W^{l,q}(M)$, we may apply the same reasoning to obtain positive function $u_3 \in W^{k+2,q}(M)$ and a diffeomorphism $\Phi_3 \in W^{k+3,q}(M,M)$ such that
    \begin{align*}
        \Phi^*_3&(u_3^2 \, \tilde g_2) = \Phi^*(u^2 g) \in W^{k+3,q}(M)
    \end{align*}
    with $\Phi = \Phi_1 \circ \Phi_2 \circ \Phi_3 \in W^{k+1,q}(M,M)$ and $u = (u_3 \circ \Phi_2^{-1} \circ \Phi_1^{-1}) (u_2 \circ \Phi_1^{-1}) u_1 \in W^{k,q}(M)$.

    Finally, if $l = k$, we name $\tilde g_3 = \Phi_3^*(u^2_3g)$ and use that $\B_{\tilde g_3} = \Phi_3^*(\B_{g_3}) \in W^{k,q}(M)$ to run the argument one last time and establish
    \begin{align*}
        \Phi^*_4&(u_4^2 \, \tilde g_3) = \Phi^*(u^2 g) \in W^{k+4,q}(M)
    \end{align*}
    with $\Phi = \Phi_1 \circ \Phi_2 \circ \Phi_3 \circ \Phi_4 \in W^{k+1,q}(M,M)$ and $u = (u_4 \circ \Phi_3^{-1} \circ \Phi_2^{-1} \circ \Phi_1^{-1}) (u_3 \circ \Phi_2^{-1} \circ \Phi_1^{-1}) (u_2 \circ \Phi_1^{-1}) u_1 \in W^{k,q}(M)$
    for some positive function $u_4 \in W^{k+3,q}(M)$ and global diffeomorphism $\Phi_4 \in W^{k+4,q}(M,M)$.
\end{proof}  
The proofs of \cref{thm: main Cotton} and \cref{thm: main Weyl} follow along the same lines as the proof of \cref{thm: main Bach}, but using the equations \eqref{eq: Cotton} and \eqref{eq: Weyl} from \cref{lemma: conformal PDEs lowreg} instead. We will thereby omit them. 

At this point, we can establish \cref{cor: Yamabe with Cotton} as well:
\begin{proof}[Proof of \cref{cor: Yamabe with Cotton}]
    By \cref{thm: main Cotton} 
     there exists a positive function $u \in W^{2,q}(M)$ and a $W^{3,q}$ diffeomorphism $\Phi : M \to M$ such that $\Phi^*(u^2g) \in W^{3,q}(M)$. Since $W^{3,q}(M) \hookrightarrow W^{2,2q}(M)$ with $2q > 3$, the assumptions of the current theorem enable to apply \cite[Theorem A]{AvalosCogoRoyo1}. Hence, there exists 
     a conformal metric in $\llbracket\,\Phi^*(u^2g)\,\rrbracket_{W^{2,q}} = \llbracket\,g\,\rrbracket_{W^{2,q}}$ with constant scalar curvature.
\end{proof}

\subsection[Proof of the main theorem]{Proof of the main theorem}
\label{section: Proof of the main theorem}

Finally, we prove our main theorem concerning the regularity of critical points of the Weyl energy \eqref{eq: Weyl energy} in four dimensions:
\begin{proof}[Proof of \cref{thm: main}]
    Let us first show that the Weyl energy \eqref{eq: Weyl energy} in the space of $W^{2,q}(M)$ Riemannian metrics is $C^1$ with Frechet derivative $D\mathbf W(g) = \B_g$. Fix a Riemannian metric $g \in W^{2,q}(M)$ and consider a sequence of Riemannian metrics $(g_i)_i \subset C^\infty(M)$ such that $g_i \to g$ in $W^{2,q}(M)$. For any smooth symmetric tensor field $h$ with $\|h\|_{W^{2,q}(M)} = 1$ and $t > 0$, we estimate
    \begin{align*}
        \frac{\bigl|\mathbf W(g + th) - \mathbf W(g) - \langle\B_g, th\rangle\bigr|}{t} &\leq \frac{\bigl|\mathbf W(g + th) - \mathbf W(g_i + th)\bigr|}{t} + \frac{\bigl|\mathbf W(g) - \mathbf W(g_i)\bigr|}{t}
        \\
        &\hspace{-2cm} + \frac{\bigl|\langle \B_g - \B_{g_i}, th\rangle\bigr|}{t} + \frac{\bigl|\mathbf W(g_i + th) - \mathbf W(g_i) - \langle\B_{g_i}, th\rangle\bigr|}{t} \,.
    \end{align*}
    Let $k \geq 4$ be a large integer. Since the Weyl energy is a $C^1$ functional in the space of $C^4$ Riemannian metrics with $D\mathbf W(g_i) = \B_{g_i}$ as in \cite[§ 4.76]{BesseBook}, the last term goes to zero as $t \to 0$ uniformly in $i$. Thus, for $t > 0$ sufficiently small, there holds
    \begin{equation*}
        \frac{\bigl|\mathbf W(g_i + th) - \mathbf W(g_i) - \langle\B_{g_i}, th\rangle\bigr|}{t} < \frac{\varepsilon}{4} \,.
    \end{equation*}
    It follows from \cref{lemma: curvature continuity} that the maps
    \begin{align*}
        W^{2,q}(M) &\to \nR \,, \quad g \mapsto \mathbf W(g)
        \\
        W^{2,q}(M) &\to W^{-2,q}(M) \,, \quad g \mapsto \B_g
    \end{align*}
    are continuous and thus we can choose $i$ sufficiently large so that
    \begin{equation*}
        \bigl|\mathbf W(g + th) - \mathbf W(g_i + th)\bigr| + \bigl|\mathbf W(g) - \mathbf W(g_i)\bigr| + \bigl|\langle \B_g - \B_{g_i}, th\rangle\bigr| < \frac{\varepsilon}{4} t \,.
    \end{equation*}
    It follows that for any $\varepsilon > 0$, we may choose $t > 0$ sufficiently small so that
    \begin{equation*}
        \frac{\bigl|\mathbf W(g + th) - \mathbf W(g) - \langle\B_g, th\rangle\bigr|}{t} < \varepsilon \,.
    \end{equation*}
    It is immediate to see by a simple density argument that the same holds true for $h \in W^{2,q}(M)$, proving that $\mathbf W$ is Frechet differentiable at $g$ with $D\mathbf W(g) = \B_g$. Furthermore, it is $C^1$ due to the continuity of $g \mapsto \B_g$. In particular, $W^{2,q}(M)$ critical metrics of the Weyl functional satisfy $\B_g = 0$ in $W^{-2,q}(M)$.
    
    At this point, we can argue as in the proof of \cref{thm: main Bach}, with the difference that the bootstrap procedure does not stop, due to the fact that
    \begin{equation*}
        \B_{\Phi^*(u^2g)} = \Phi^*(u^2\B_g) = 0
    \end{equation*}
    is $C^\infty(M)$ regardless of the regularity of the diffeomorphism $\Phi$. We conclude that if $g \in W^{2,q}(M)$ is a critical point of the Weyl energy \eqref{eq: Weyl energy}, for any given $k \geq 3$ there exists a positive function $u_1 \in W^{2,q}(M)$ and diffeomorphism $\Phi_1 \in W^{3,q}(M,M)$ such that $\Phi_1^*(u_1^2g) \in C^k(M)$. In particular, for some fixed $k$ sufficiently large we can solve the Yamabe problem and find a positive function $u_2 \in C^k(M)$ such that $\tilde g = u_2^2 \Phi_1^*(u_1^2g)$ has constant scalar curvature. Note that $\B_{\tilde g} \equiv 0$ too. Then, we consider the elliptic system
    \begin{align*}
        \Delta_{\tilde g} \Ric_{\tilde g} &= \Rm_{\tilde g}*\Ric_{\tilde g} 
        \\
        \Delta_{\tilde g} \tilde g_{ij} &= -2\widetilde{\R}_{ij} + Q_{ij}( \tilde g, \partial \tilde g )
    \end{align*}
    in $\tilde g$-harmonic coordinates, where  $Q_{ij}$ is quadratic in $\partial \tilde g$. 
    We can bootstrap via \cref{thm: Laplace interior elliptic regularity} just as before to establish that $\tilde g_{ij} \in C^{\infty}_{loc}$ in $\tilde g$-harmonic coordinates. By \cref{prop: harmonic atlas}, the $\tilde g$-harmonic charts generate a $C^\infty$ differential structure on $M$, and \cref{thm: Whitney Sobolev} grants the existence of a global diffeomorphism $\Phi_2 \in C^{k+1}(M,M)$ such that $\Phi_2^*\tilde g \in C^\infty(M)$. We deduce that
    \begin{equation*}
        \Phi^*(u^2g) \in C^\infty(M)
    \end{equation*}
    with $\Phi = \Phi_1 \circ \Phi_2 \in W^{3, q}(M, M)$ and $u = (u_2 \circ \Phi_1^{-1}) u_1 \in W^{2, q}(M)$. 
\end{proof}
Finally, the proof of \cref{thm: main Weyl-flat} is similar to the second part of the proof of \cref{thm: main}, using the equation \eqref{eq: Weyl} from \cref{lemma: conformal PDEs lowreg} instead.

\vspace{0.5cm}

\printbibliography[]

\clearpage

\end{document}